\documentclass[leqno]{article}
\usepackage{geometry}
\usepackage{graphicx}	
\usepackage[utf8]{inputenc}
\usepackage[T2A]{fontenc}
\usepackage[russian]{babel}
\usepackage{mathtools}
\usepackage{amsfonts,amssymb,mathrsfs,amscd,amsmath,amsthm}

\usepackage{verbatim}

\usepackage{url}

\ifpdf
\usepackage{hyperref}
\else
\usepackage[dvipdfmx]{hyperref}
\fi

\usepackage{stmaryrd}

\def\thtext#1{
\catcode`@=11
\gdef\@thmcountersep{. #1}
\catcode`@=12
}

\def\threst{
\catcode`@=11
\gdef\@thmcountersep{.}
\catcode`@=12
}

\theoremstyle{plain}

\newtheorem{thm}{ Theorem }[section]

\newtheorem{cor}[thm]{ Corollary }

\newtheorem{lem}[thm]{ Lemma }

\theoremstyle{definition}

\newtheorem{rk}[thm]{ Note }

\catcode`@=11
\def\.{.\spacefactor\@m}
\catcode`@=12

\def\R{\mathbb R}

\def\Z{{\mathbb Z}}

\def\a{\alpha}

\def\e{\varepsilon}
\def\dl{\delta}
\def\D{\Delta}
\def\g{\gamma}

\def\l{\lambda}

\def\0{\emptyset}
\def\:{\colon}
\def\<{\langle}
\def\>{\rangle}
\def\[{\llbracket}
\def\]{\rrbracket}

\def\d{\partial}

\def\rom#1{\emph{#1}}
\def\({\rom(}
\def\){\rom)}
\def\sm{\setminus}
\def\ss{\subset}

\def\sp{\supset}

\def\x{\times}

\def\bX{{\bar X}}

\def\bY{{\bar Y}}

\def\diam{\operatorname{diam}}

\def\dis{\operatorname{dis}}

\def\GH{\operatorname{\mathcal{G\!H}}}

\def\Int{\operatorname{Int}}

\def\RP{\operatorname{\RP}}

\def\cC{{\cal C}}

\def\cE{{\cal E}}

\def\cG{{\cal G}}

\def\cL{{\cal L}}

\def\cR{{\cal R}}

\usepackage{epigraph}

\begin{document}
\title{Gromov--Hausdorff Geometry of Metric Trees}
\author{A.\,O.~Ivanov, I.\,N.~Mikhailov, A.\,A.~Tuzhilin}
\date{}
\maketitle

\begin{abstract}
In this paper, we study metric trees, without any finiteness restrictions. For subsets of such trees, a condition that guarantees that the Hausdorff and Gromov--Hausdorff distances from the subset to the entire metric tree are the same is obtained. This result allows to construct a new class of shortest geodesics (in the proper class of all metric spaces) connecting such subset of a metric tree with the tree itself. In particular, the technique elaborated is demonstrated on subsets of the real line.

\end{abstract}

\setlength{\epigraphrule}{0pt}

\section{Introduction}
\markright{\thesection.~Introduction }
This paper is devoted to geometry of the Gromov--Hausdorff distance in the case of a metric tree cloud, i.e., a proper class of metric spaces at a finite Gromov--Hausdorff distance from a fixed metric tree (that need not be finite). As a particular case the cloud of the the real line is considered.

The \emph{Gromov--Hausdorff distance\/} measures a degree of dissimilarity of metric spaces and is defined as the half of the infimum of distortions of all possible correspondences (i.e., multivalued surjective mappings) between metric spaces. The distortion of a correspondence is the supremum of the absolute differences of distances between point pairs of one  space and a pair of their images in the other space. Evidently, the Gromov--Hausdorff distance is zero between isometric spaces, although it can be equal to zero between non-isometric ones too.  Moreover, the Gromov-Hausdorff distance can also be equal to infinity, for example, between a bounded and an unbounded spaces. The details concerning the Gromov--Hausdorff distance can be found in~\cite{BurBurIva}.

The family of all metric spaces forms a proper class in the sense of the von Neumann--Bernays--G\"odel set theory. Metric spaces considering up to isometries, form a proper class too, which we denote by $\GH$ and call the \emph{Gromov--Hausdorff class}. However, even after such factorization, there remain pairs of different elements from $\GH$, the distance between which is zero. Even if one identifies the metric spaces at zero distance, the result is a proper class as well, and we denote it by $\GH_0$. On this class the Gromov--Hausdorff distance is positively defined. Finally, we partition $\GH_0$ into maximal subclasses consisting of spaces on finite distance. We call these subclasses \emph{clouds}, and thus, the Gromov--Hausdorff distance is a metric on each such cloud. If $X$ is a metric space, then we denote by $[X]$ the unique cloud containing $X$. In this paper, we will be especially interested in properties of the cloud $[\R]$ of the real line $\R$ and, more generally, of the ones containing metric trees.

When studying clouds, the mapping $\GH_0\x\R_{>0}\to\GH_0$, which takes a metric space $X$ and a positive real number $\l$ to the metric space $\l X$ obtained from $X$ by multiplying all distances in $X$ by $\l$. It turns out that this mapping, being a contraction on the cloud of bounded metric spaces, in general, possesses several unexpected properties. S.~Bogaty and A.~Tuzhilin~\cite{BogatyTuzhilin} has found that the distance between spaces $X$ and $\l X$ can be infinite, i.e., multiplication by $\l$ could take a metric space $X$ to another cloud. But in the case of the cloud $[\R]$ each such multiplication maps the cloud into itself. However, as I.~Mikhailov shown~\cite{MikhailovZ}, multiplication by $\l$ is not continuous in general: if we multiply integers $\Z\in[\R]$ by any $\l>1$, then the Gromov--Hausdorff distance between $\Z$ and the resulting space $\l\Z$ is at least $1/2$. Thus, even in this case the multiplication by $\l$ is not a contraction.

However, as it turns out, there is another way to ``transform'' some elements of this cloud to $\R$. Namely, here we show that for each $\e$-networks $S$ in $\R$, the canonical Hausdorff geodesic connecting $S$ and $\R$ (see below for definitions) is, in fact, a geodesic in $[\R]$. This construction gives a hope for constructing a contraction that differs from the multiplication by $\l$. In proving this fact, we obtained another remarkable result: the distances between the real line and its subset, calculated in Hausdorff and Gromov-Hausdorff metrics, coincide with each other. Notice that the specific values of the Gromov--Hausdorff  distance, in contrast to the Hausdorff distances, is a great rarity. Also notice that the latter result is proved using an interesting lower bound for the Gromov-Hausdorff distance, found in~\cite{UltraMemoli} by means of the so-called ultrametrization, see below.

Just before the publication of this paper, the paper~\cite{MetricTrees} has appeared. It is shown in~\cite{MetricTrees} that for an arbitrary subset $X$ of a finite metric tree $T$, if the Hausdorff distance $d_H(X,\,T)$ is greater than the oriented Hausdorff distance from the boundary of the tree $T$ to $X$, then the Gromov--Hausdorff distance between $X$ and $T$ equals $d_H(X,\,T)$. We have observed that the technique of the proof that is based on Lemma~2.3 from~\cite{MetricTrees} is also closely related to the ultrametrization mapping, although this mapping is not used explicitly. In this paper, we, in particular, generalize this result from~\cite{MetricTrees} to equality of Hausdorff and Gromov--Hausdorff distances for infinite metric trees subsets.

\section{Preliminaries}\label{sec:GH}
\markright{\thesection.~Preliminaries}
Let us recall the necessary concepts and results concerning Hausdorff and Gromov--Hausdorff distances. More detailed information can be found in~\cite{BurBurIva}.

Let $X$ be a metric space and $x$ and $y$ be its points, then $|xy|$ will denote the distance between these points, and if $A$ and $B$ are non-empty subsets of $X$, then we put $|AB|=|BA|=\inf\bigl\{|ab|:a\in A,\,b\in B\bigr\}$. If $A=\{a\}$, then instead of $\bigl|\{a\}B\bigr|=\bigl|B\{a\}\bigr|$ we will write $|aB|=|Ba|$. For non-negative real $r$ we define \emph{closed $r$-neighborhood of the set $A$} as follows:
$$
B_r(A)=\{x\in X:|xA|\le r\}.
$$
Let $X$ be a metric space and $A$, $B$ be non-empty subsets of $X$. The value
$$
d_H(A,B)=\max\bigl\{\sup_{a\in A}|aB|,\,\sup_{b\in B}|Ab|\bigr\}
$$
is called the \emph{Hausdorff distance between $A$ and $B$}. Equivalently,
$$
d_H(A,B)=\inf\bigl\{r:A\ss B_r(B)\ \ \text{and}\ \ B_r(A)\sp B\bigr\}.
$$
\emph{The oriented Hausdorff distance} from $A$ to $B$ is called the value
$$
\overrightarrow{d_H}(A,\,B)=\sup_{a\in A}|aB|.
$$
We extend $\overrightarrow{d_H}(A,\,B)$ for empty $A$ as $\overrightarrow{d_H}(\0,\,B)=0$.

It is well known that $d_H$ is a generalized metric on the set $\cC(X)$ of all non-empty closed subsets of a metric space $X$. Here the word ``generalized'' means that $d_H$ can be equal to infinity on some pairs of subsets (for example, on a bounded and unbounded ones).

Next, let $X$ and $Y$ be non-empty metric spaces. A multivalued surjective mapping $R$ from $X$ to $Y$ is called \emph{correspondence between $X$ and $Y$}. The set of all correspondences between $X$ and $Y$ are denoted by $\cR(X,Y)$. For any correspondence $R\in\cR(X,Y)$, we define its \emph{distortion\/} as
$$
\dis R=\sup\Bigl\{\bigl||xx'|-|yy'|\bigr|:(x,y),\,(x',y')\in R\Bigr\}.
$$
\emph{The Gromov--Hausdorff distance\/} between $X$ and $Y$ is
$$
d_{GH}(X,Y)=\frac12\inf\bigl\{\dis R:R\in\cR(X,Y)\bigr\}.
$$
Note~\cite{BurBurIva} that this value is also equal to the infimum of the Hausdorff distances between the images of the spaces $X$ and $Y$ while their isometric embeddings into all possible metric spaces.

For a metric space $X$, we denote $\diam X$ by its \emph{diameter\/}:
$$
\diam X=\sup\bigl\{|xy|:x,y\in X\bigr\}.
$$
For a single-point metric space we reserve the notation $\D_1$. It is well known~\cite{BurBurIva} that $2d_{GH}(\D_1,X)=\diam X$.

\section{Subsets of Euclidean space at a finite Gromov--Hausdorff distance}
\markright{\thesection.~Subsets of Euclidean space at a finite Gromov--Hausdorff distance}

In the paper~\cite{TuzhilinMikhailov} a natural question was posed: to describe all subsets of a given finite-dimensional normed space at a finite Gromov--Hausdorff distance from it. We will need the main result of this work.

\begin{thm}[\cite{TuzhilinMikhailov}]\label{thm:SUbsetRnFiniteGH}
A subset $X$ of a finite-dimensional Euclidean space $\R^n$ is at a finite Gromov--Hausdorff distance from $\R^n$ if and only if $X$ is an $\e$-network for some finite $\e$. The latter is equivalent to the fact that the Hausdorff distance between $X$ and $\R^n$ is finite.
\end{thm}

In fact, the statement of Theorem \ref{thm:SUbsetRnFiniteGH} can be reformulate as follows: for an arbitrary $A\ss\R^n$ the conditions $d_{GH}(\R^n,\,A)=\infty$ and $d_H(\R^n,\,A)=\infty$ are equivalent. Thus, the question arises whether Hausdorff and Gromov--Hausdorff distances from $\R^n$ to its arbitrary subsets of $A$ are the same in the case where both of these quantities are finite. Below we will prove this statement in the case $n=1$.

\section{Ultrametrization}
\markright{\thesection.~Ultrametrization}
Let $X$ be an arbitrary metric space. The finite the sequence $L=\{x=x_0,x_1,\ldots,x_n=y\}$ of points from $X$ will be called \emph{a dotted line connecting $x$ and $y$}. We put $|L|_u:=\max_{i=1}^n\bigl\{|x_{i-1}x_i|\bigr\}$ and we call this value \emph{the ultrametric length of $L$}. The set of all dotted lines in $X$, connecting $x$ and $y$, we denote by $\cL_{x,y}$.

Following~\cite{UltraMemoli}, for each pair $x,y\in X$ we put
$$
|xy|_u=\inf\bigl\{|L|_u:L\in\cL_{x,y}\bigr\}.
$$
As shown in~\cite{UltraMemoli}, the function $|\cdot\cdot|_u$ on $X\x X$ is a pseudometric. Let $U(X)$ denote the quotient space obtained from $X$ by factorization by zeros of the pseudometric $|\cdot\cdot|_u$. The metric space $U(X)$, as well as the natural projection $\pi\:X\to U(X)$ are called \emph{ultrametrizations}.

The following inequality was first formulated in the work \cite{UltraMemoli} for finite metric spaces and subsequently generalized to the case of bounded~\cite{LMZ} and arbitrary~\cite{MikhailovUltra} metric spaces.

\begin{thm}[\cite{UltraMemoli},\cite{LMZ},\cite{MikhailovUltra}]\label{thm:UltraDownEst}
For any non-empty metric spaces $X$ and $Y$, we have the following inequality\/\rom:
$$
d_{GH}(X,Y)\ge d_{GH}\bigl(U(X),U(Y)\bigr).
$$
\end{thm}

\begin{cor}\label{thm:UltraDownZero}
If metric spaces $X$ and $Y$ are such that $d_{GH}(X,Y)=0$, then $d_{GH}\bigl(U(X),U(Y)\bigr)=0$. In particular, this is true for non-empty subsets $X$ and $Y$ of a metric space $Z$, for which $d_H(X,Y)=0$, where the latter is equivalent to the coincidence of the closures $\bX$ and $\bY$.
\end{cor}

Let's apply the triangle inequality and the corollary~\ref{thm:UltraDownZero}.

\begin{cor}\label{cor:DiamsUltra}
Let $X$, $Y$, and $Z$ be nonempty metric spaces and $d_{GH}(X,Y)=0$, or $X$ and $Y$ are non-empty subsets of the metric spaces $W$ such that $d_H(X,Y)=0$ \(in other words, the closures of $X$ and $Y$ coincide\/\). Then $d_{GH}(Z,X)=d_{GH}(Z,Y)$ and $d_{GH}\bigl(Z,U(X)\bigr)=d_{GH}\bigl(Z,U(Y)\bigr)$. In particular, this holds for $Z=\D_1$, so $\diam X=\diam Y$ and $\diam U(X)=\diam U(Y)$.
\end{cor}

A metric space $X$ is called \emph{dotted connected\/} if $U(X)=\D_1$. The latter means that each pair of points from $X$ can be connected by a dotted line $L$ with an arbitrarily small ultrametric length.

\begin{rk}[\cite{MikhailovUltra}]\label{rk:pathconnectedtodelta1}
Any path-connected metric space $X$ is dotted connected.
\end{rk}

Note that not every dotted connected space is path-connected, for example, the set of rational points of the segment $[0,\,1]$.

\section{Kuratowski's Embedding}
\markright{\thesection.~Kuratowski's Embedding}

Let $X$ be an arbitrary metric space. By $C_b(X)$ we denote the Banach space of all bounded continuous real-valued functions on $X$ with $\sup$-norm.

\begin{thm}[\cite{Kuratowski}]
For an arbitrary metric space $X$ and a fixed point $x_0\in X$, the mapping
$$
\Phi\:X\to C_b(X),\;\Phi(x)(y)=d_X(x,\,y)-d_X(x_0,\,y)\quad\forall\,x,\,y\in X
$$
is isometric.
\end{thm}

We will need the following construction from~\cite{MikhailovUltra}. Consider a metric space $X$ and an arbitrary real $t>0$. Let $\Phi\:X\to C_b(X)$ be the Kuratowski embedding. Add to $\Phi(X)$ all segments in $C_b(X)$ with endpoints at $\Phi(x)$, $\Phi(y)$, for which $d_X(x,\,y)\le t$. By $D_t(X)$ we denote the resulting subset of $C_b(X)$ with the induced metric. We also define $D_0(X)$, setting it equal to $\Phi(X)$.

\begin{lem}[\cite{MikhailovUltra}]\label{lem: Dc(X)}
Let $X$ be a metric space. Then
\begin{enumerate}
\item if $\diam U(X)<t<\infty$, then the space $D_t(X)$ is path-connected\/\rom;
\item if $\diam U(X)=0$, then the space $D_0(X)$ is dotted connected\/\rom;
\item for an arbitrary $t\ge 0$, it holds $d_{GH}\bigl(X,\,D_t(X)\bigr)\le t/2$.
\end{enumerate}
\end{lem}

\begin{cor}\label{cor:pathconnectedenseinBLconnected}
The closure of the class of all path-connected metric spaces in the Gromov--Hausdorff class coincides with the class of all dotted connected metric spaces.
\end{cor}

\begin{proof}
First, let $X$ be a dotted connected metric space. Then by Lemma~\ref{lem: Dc(X)} the space $D_t(X)$ is path-connected for each $t>0$, and $d_{GH}\bigl(X,\,D_t(X)\bigr)\le t/2$. Thus, $X$ is contained in the closure of the family of all path-connected metric spaces.

Now suppose that the sequence $X_n$ of path-connected metric spaces converges in the Gromov-Hausdorff sense to a metric space $X$. Then by Theorem~\ref{thm:UltraDownEst} we have
$$
d_{GH}(X_n,\,X)\ge d_{GH}\bigl(U(X_n),\,U(X)\bigr)=d_{GH}\bigl(\D_1,\,U(X)\bigr)=\frac12\diam U(X).
$$
Since $d_{GH}(X_n,\,X)\to 0$ as $n\to\infty$, we obtain that $\diam U(X)=0$, that is, $X$ is dotted connected.
\end{proof}

\section{Metric trees}
\markright{\thesection.~Metric trees}
Let $\{e_\a\}_{\a\in I}$ be an arbitrary family of segments of the real line. Let us set $\cE=\sqcup_{\a\in I}e_\a$ and on the obtained topological space we introduce an arbitrary equivalence relation $\sim$, identifying some ends of the segments $e_\a$. The quotient space $\cG=\cE/\!\sim$ is called a \emph{topological graph}. If $\pi\:\cE\to\cE/\!\sim$ is the canonical projection, then the $\pi$-images the endpoints of the segments $e_\a$ are called \emph{the vertices of $\cG$}, and the $\pi$-images of the segments $e_\a$ \emph{the edges $\cG$}. The terminology of the graph theory is directly transferred to topological graphs, in particular, for they there are defined the degrees of vertices, the paths, the cycles, the concept of connectivity, of connected component, etc. A topological graph $\cG$ is called a \emph{tree}, if the graph is connected and does not contain cycles. Note that for a connected graph $\cG$, the standard distance functions on segments $e_\a$ is naturally extended to the whole $\cG$ and generates an intrinsic pseudometric, see for details~\cite{BurBurIva}. A topological tree with such distance function is be called a \emph{metric tree}.

A metric tree is called \emph{finite\/} if it contains a finite number of edges. \emph{The boundary $\d T$} of a finite metric tree $T$ is the subset of $T$ consisting of all vertices of $T$ of degree $1$. The next result was obtained in~\cite{MetricTrees} for finite metric trees.

\begin{thm}\label{thm:CompactMetricTrees}
Let $T$ be a finite metric tree. Then for any non-empty subsets $X\ss T$, the inequality $d_H(T,\,X)>\overrightarrow{d_H}(\d T,\,X)$ implies $d_{GH}(T,\,X)=d_H(T,\,X)$.
\end{thm}

The proof of this theorem was based mainly on the following lemma.

\begin{lem}\label{lem:ImageofConnected}
Let $X$ be a subset of a metric space $G$ with intrinsic metric, $R\in\cR(G,\,X)$, $\dis R<r$. Then for any connected subset $S\ss G$, the set $B_r\bigl(R(S)\bigr)$ is connected.
\end{lem}

Below we will show that the technique by which the Theorem~\ref{thm:CompactMetricTrees} was proved is closely related to the ultrametrization, if instead of connectivity we consider the linear connectivity. We will also show that the statement of Theorem~\ref{thm:CompactMetricTrees} can be transferred to the case of infinite metric trees without serious changes.

\section{Canonical Hausdorff geodesics}
\markright{\thesection.~Canonical Hausdorff geodesics}
Let $X$ be an arbitrary proper geodesic metric space, and suppose that $A$ and $B$ are non-empty closed subsets of $X$ at a finite Hausdorff distance from each other: $d=d_H(A,B)<\infty$. For each $t\in[0,d]\ss\R$, we put $C_t=B_t(A)\cap B_{dt}(B)$. Then each $C_t$ is closed in $X$, i.e\. $C_t\in\cC(X)$ for all $t$.

\begin{thm}\label{thm:HausdorffGeodesic}
In the notation introduced above, the sets $C_t$ are non-empty, and the curve $t\mapsto C_t$ is a shortest geodesic in $\cC(X)$ such that its length equals to the distance between its ends.
\end{thm}

\begin{proof}
Let us show that $C_t$ are non-empty. We choose an arbitrary point $a\in A$ and positive $\e$, then the set $D:=B_{d+\e}(a)\cap B$ is non-empty and is compact, so there exists $b\in D$ at which the distance from $a$ to $D$ is achieved. Since $\bigl|a\,(B\sm D)\bigr|\ge d+\e$, and $|aD|\le d+\e$, then
$$
|aB|=\min\Bigl\{|aD|,\bigl|a\,(B\sm D)\bigr|\Bigr\}=|aD|.
$$
By definition of the Hausdorff distance, $|aB|<d+\dl$ for all $\dl>0$, therefore $r:=|ab|=|aD|\le d$.

Let us connect $a$ and $b$ by a shortest naturally parameterized geodesic $\g\:[0,r]\to X$, then $r\le d$, so there exists $t_0\in[0,r]$, for which $\bigl|a\g(t_0)\bigr|\le t$ and $\bigl|\g(t_0)b\bigr|\le dt$. But then $\g(t_0)\in C_t$, which proves that $C_t$ is non-empty.

Let us now show that for $0\le t\le s\le d$ it holds $d_H(C_t,C_s)\le st$. To do this, it is enough to check that for each $x\in C_t$ we have $|xC_s|\le st$ and for each $y\in C_s$ it holds $|yC_t|\le s-t$. Since both statements are proven in exactly the same way, let's check the first of them only.

By construction, there exist $a\in A$ and $b\in B$ for which $|xa|\le t$ and $|xb|\le dt$. Let us connect $x$ and $b$ by a shortest geodesic and choose on this geodesic a point $z$, for which $|zb|=\min\bigl\{d-s,|xb|\bigr\}$. Then $|xz|\le st$, whence, by virtue of the triangle inequality, $|az|\le s$. Therefore, $z\in C_s$ and $|xz|\le st$, that was required. Now the assertion of the theorem follows from the triangle inequality for the Hausdorff distance.
\end{proof}

The curve from Theorem~\ref{thm:HausdorffGeodesic} is called \emph{the canonical Hausdorff geodesic}.

\section{Infinite metric trees}
\markright{\thesection.~Infinite metric trees}
Let $X$ be an arbitrary metric space. We put
\begin{align*}
t(X)&=\inf\bigl\{d_{GH}(X,\,Y)\: Y\,\text{is path-connected}\bigr\}, \\
\tilde{t}(X)&=\inf\bigl\{d_{GH}(X,\,Y)\: Y\,\text{dotted connected}\bigr\}.
\end{align*}

\begin{thm}
The equalities $2t(X)=2\tilde{t}(X)=\diam U(X)$ are satisfied.
\end{thm}

\begin{proof}
The first equality follows from Corollary~\ref{cor:pathconnectedenseinBLconnected}. Let us justify the second equality.

To start with, suppose that $\diam U(X)=\infty$. Assume that $\tilde{t}(X)<\infty$, i.e., there is a dotted connected space $Y$ at a finite Gromov--Hausdorff distance from $X$. By Theorem~\ref{thm:UltraDownEst} we have $d_{GH}(X,\,Y)\ge d_{GH}\bigl(U(X),\,U(Y)\bigr)$. However, $U(Y)=\D_1$ by definition, therefore $d_{GH}\bigl(U(X),\,U(Y)\bigr)=\frac12\diam U(X)=\infty$, whence $d_{GH}(X,\,Y)=\infty$, a contradiction.

Now let $c:=\diam U(X)<\infty$. For each $t>c$ we construct the space $D_t(X)$ described before Lemma~\ref{lem: Dc(X)}. By the same lemma, the space $D_t(X)$ is path-connected and $d_{GH}\bigl(X,D_t(X)\bigr)\le t/2$, whence $\tilde{t}(X)=t(X)\le c/2$ in due to the arbitrariness of $t$. On the other hand, Theorem~\ref{thm:UltraDownEst} implies that for any path-connected metric space $Y$ the following inequality holds:
$$
d_{GH}(X,\,Y)\ge d_{GH}\bigl(U(X),\,U(Y)\bigr)=d_{GH}\bigl(U(X),\,\D_1\bigr)=c/2.
$$
Passing to the infimum, we obtain $\tilde{t}(X)\ge c/2$, which completes the proof.
\end{proof}

Let $X$ be a non-empty subset of a metric tree $T$. A point $a\in T\sm X$ is called \emph{internal with respect to $X$} if in the tree $T$ there exists a path (topological embedding of a segment) with ends in $X$, passing through $a$. The set of all such points $a$ is denoted by $\Int_XT$, and the remaining points from $T\sm X$ we call \emph{end points with respect to $X$} and their set we denote by $\d_XT$. If $T$ is a finite tree, then points from $(\d T)\sm X$ are contained in $\d_XT$ and are the most farthest from $X$ among all points in $\d_XT$; points in $(\d T)\cap X$ are at zero distance from $X$, so for finite trees we have $\overrightarrow{d_H}(\d T,\,X)=\overrightarrow{d_H}(\d_XT,\,X)$ (recall that the oriented Hausdorff distance from the empty set we defined to be equal to zero). For infinite trees, the set $\d T$ can be empty, but $\d_XT$ may not be empty. For example, if we glue the origins of the three rays together, and divide the rays into segments, thus obtaining an infinite tree $T$, and we take as $X$ the union of all the segments of two rays, then the open third ray will be non-empty $\d_XT$, however, in this case $\d T=\0$.

\begin{thm}\label{thm:infinitemetrictrees}
Let $T$ be an arbitrary metric tree, $X$ be non-empty subset of $T$ such that either $\d_XT=\0$, or $\d_XT\ne\0$ and $d_H(X,\,T)>\overrightarrow{d_H}(\d_XT,\,X)$. Then
$$
d_{GH}(X,\,T)=d_H(X,\,T).
$$
\end{thm}

\begin{proof}
It is known that $d_{GH}(X,\,T)\le d_H(X,\,T)$, so it suffices to prove the opposite inequality. If $d_H(X,\,T)=0$, then the reverse inequality is obvious. Now let $d_H(X,\,T)\ne0$.

Let $c=\diam U(X)$. It follows from Theorem~\ref{thm:UltraDownEst} that
$$
d_{GH}(X,\,T)\ge d_{GH}\bigl(U(X),\,U(T)\bigr)=d_{GH}\bigl(U(X),\,\D_1\bigr)=c/2,
$$
therefore, it suffices to justify the equality $c=2d_H(X,\,T)$.

To start with, let $d_H(X,\,T)=\infty$. We must show that $\diam U(X)=\infty$. Let's consider two cases.

(1) Suppose that $\d_XT=\0$, then each point from $T\sm X$ is internal w.r.t\. $X$. On the other hand, due to $d_H(X,\,T)=\infty$, for any $d>0$ there exists $a\in T\sm X$ such that $|aX|>d$. Since $\d_XT=\0$, then in each of the connected components of the space $T\sm\{a\}$ there is a point from $X$. Let's choose any two of these components, and in them arbitrary points $x$ and $x'$, respectively. Consider any dotted line $L=\{x_0=x,x_1,\ldots,x_n=x'\}$, where $x_i\in X$ for all $i$. Then there exist $x_{i-1}$ and $x_i$ lying in different components. Since $|aX|>d$, then $|x_{i-1}x_i|>2d$, whence $|L|_u\ge2d$. By virtue of arbitrariness of the dotted line $L$, we have $|x_{i-1}x_i|\ge2d$, whence $\diam U(X)\ge2d$ and, therefore, $\diam U(X)=\infty$, because $d$ is arbitrary.

(2) Now suppose that $\d_XT\ne\0$, then $d_H(X,\,T)>\overrightarrow{d_H}(\d_XT,\,X)$, from which the value $h=\overrightarrow{d_H}(\d_XT,\,X)$ is finite. Since $d_H(X,\,T)=\infty$, then there exists $a\in T\sm X$ for which $|aX|>h$ and, therefore, the point $a$ is internal w.r.t\. $X$. It remains to repeat the reasoning from item $(1)$.

Now let $0<d_H(X,\,T)<\infty$. We already know that $d_H(X,\,T)\ge d_{GH}(X,\,T)\ge c/2$. We will show that $d_H(X,\,T)\le c/2$. Note that
$$
d_H(X,\,T)=\sup_{a\in T\sm X}|aX|=\max\biggl\{\sup_{a\in\Int_XT}|aX|,\sup_{a\in\d_XT}|aX|\biggr\}.
$$
If $\d_XT=\0$, then $\Int_XT\ne\0$, because otherwise $d_H(X,T)=0$. Therefore
$$
d_H(X,\,T)=\sup_{a\in\Int_XT}|aX|=\overrightarrow{d_H}(\Int_XT,\,X).
$$
If $\d_XT\ne\0$, then $d_H(X,\,T)>\overrightarrow{d_H}(\d_XT,\,X)$, so $\Int_XT\ne\0$ and again
$$
d_H(X,\,T)=\overrightarrow{d_H}(\Int_XT,\,X).
$$

Since $d_H(X,\,T)=\sup_{a\in T\sm X}|aX|<\infty$, then there exists $a\in\Int_XT$, for which $\bigl||aX|-d_H(X,\,T)\bigr|\le\e$. Since $a\in\Int_XT$, then there is a path $\g\:[0,\,1]\to T$ such that $\g(0)\in X$, $\g(1)\in X$, and $\g(t)=a$ for some $t\in(0,\,1)$. By definition of ultrametrization, for any $\e>0$ there is a dotted line $x_0=\g(0),\,x_1,\,\ldots,\,x_n=\g(1)$ with vertices in $X$ such that $\max_{0\le i\le n-1}|x_ix_{i+1}|\le c+\e$. Since $\g$ is an embedding, $a$ is an interior point of the curve $\g$, and $T$ is a tree, then the points $x$ and $x'$ lie in different path-connected components of the space $T\sm\{a\}$. This means that there exists an index $i$ such that $a$ lies on the only path in $T$ connecting points $x_i$ and $x_{i+1}$. Then $|aX|\le|x_ix_{i+1}|/2\le(c+\e)/2$. Since $\bigl||aX|-d_H(X,\,T)\bigr|\le\e$, we obtain that $d_H(X,\,T)-\e\le|aX|\le(c+\e)/2$. The desired inequality follows from the arbitrariness of $\e>0$.
\end{proof}

\section{New geodesics in Gromov--Hausdorff class}
\markright{\thesection.~New geodesics in Gromov--Hausdorff class}

\begin{thm}\label{thm:HGHgeodesic}
Let $X$ be a proper geodesic metric space and $A$ a closed subset of $X$. Suppose that $d_H\bigl(A,\,X\bigr)=d_{GH}\bigl(A,\,X\bigr)<\infty$. Then the canonical Hausdorff geodesic $t\mapsto C_t$ connecting $A$ and $X$ is the shortest curve in the cloud $[X]$.
\end{thm}

\begin{proof}
Let $t$ run through the segment $[0,g]$, $C_0=A$, $C_g=X$. We choose arbitrary $0\le t_1<t_2\ldots<t_n\le g$. By Theorem~\ref{thm:HausdorffGeodesic} we have $\displaystyle d_H(A,X)=\sum_{i = 1}^{n-1}d_H(A_{t_i},\,A_{t_{i+1}})$, therefore
$$
d_{GH}(A, X)\le\sum_{i=1}^{n-1}d_{GH}(A_{t_i},\,A_{t_{i+1}})\le \sum_{i=1}^{n-1}d_{H}(A_{t_i},\,A_{t_{i+1}})=d_H(A,\,X)=d_{GH}(A,X),
$$
hence $\displaystyle d_{GH}(A,\,X)=\sum_{i=1}^{n-1}d_{GH}(A_{t_i},\,A_{t_{i+1}})$. It follows that the length of the curve
$t\mapsto C_t$ in the cloud $[X]$ is equal to the distance between its ends, i.e., this curve is a shortest geodesic in the cloud.
\end{proof}

\begin{thm}\label{thm:HausdEqGromovHausdR}
Let $X$ be a subset of the real line $\R$, then $d_{GH}(X,\R)=d_H(X,\R)$.
\end{thm}

\begin{proof}
By Theorem~\ref{thm:SUbsetRnFiniteGH} the set $X$ is at a finite Gromov--Hausdorff distance from $\R$ if and only if $d_H(X,\R)<\infty$. Thus, if one of the quantities $d_{GH}(X,\R)$ or $d_H(X,\R)$ is infinite, then the second one is too. In other words, in the case for infinite distances the statement of the theorem is true.

Now let both of these quantities be finite. Note that $\R$ can be considered as an infinite metric tree $T$, for which the set of vertices coincides with $\Z$. Since $d_H(X,\,\R)<\infty$, then each point $a\in T\sm X$ is internal w.r.t\. $X$, i.e., $\d_XT=\0$. Then the equality $d_{GH}(X,\,\R)=d_H(X,\,\R)$ follows from Theorem~\ref{thm:infinitemetrictrees}.
\end{proof}

\begin{cor}\label{cor:Rgeodesic}
Let $X\ss\R$ be a non-empty closed subset such that $d_{GH}(X,\R)<\infty$. Then the canonical Hausdorff geodesic connecting $X$ and $\R$ is a shortest curve in the cloud $[\R]$.
\end{cor}

\begin{proof}
It follows from Theorem~\ref{thm:HausdEqGromovHausdR} and Theorem~\ref{thm:HGHgeodesic}.
\end{proof}

\begin{cor}\label{cor:infinitetreecloud}
Let $T$ be a metric tree whose vertex degrees are finite, $X$ be a non-empty closed subset of $T$ such that either $\d_XT=\0$, or $\d_XT\ne\0$ and $d_H(X,\,T)>\overrightarrow{d_H}(\d_XT,\,X)$. Then the canonical Hausdorff geodesic connecting $X$ to $T$ is a shortest curve in the cloud $[T]$.
\end{cor}

\begin{proof}
Since the degree of each vertex $T$ is finite, then $T$ is proper. Since $T$ is also geodesic, the statement follows from Theorem~\ref{thm:infinitemetrictrees} and Theorem~\ref{thm:HGHgeodesic}.
\end{proof}

Denote by $\R^n_\infty$ the space $\R^n$ endowed with $\ell_\infty$-norm. Let $\Z^n_\infty$ be the integer lattice in $\R^n_\infty$.

\begin{thm}\label{thm:RninftyZninfty}
It holds $d_H(\R^n_\infty,\,\Z^n_\infty)=d_{GH}(\R^n_\infty,\,\Z^n_\infty)$.
\end{thm}

\begin{proof}
Since $d_{GH}(\R^n_\infty,\,\Z^n_\infty)\le d_H(\R^n_\infty,\,\Z^n_\infty)=\frac12$, it suffices to prove the opposite inequality.

By Theorem \ref{thm:UltraDownEst} we have $d_{GH}(\R^n_\infty,\,\Z^n_\infty)\ge d_{GH}\bigl(U(\R^n_\infty),\,U(\Z^n_\infty)\bigr)$. Note that $U(\R^n_\infty)=\D_1$, and $U(\Z^n_\infty)$ is a countable simplex of diameter $1$. Hence, $d_{GH}\bigl(U(\R^n_\infty),\,U(\Z^n_\infty)\bigr)=\frac12$, which completes the proof.
\end{proof}

\begin{cor}
Canonical Hausdorff geodesic connecting $\Z^n_\infty$ and $\R^n_\infty$ is a shortest path in the cloud $[\R^n_\infty]$.
\end{cor}

\begin{proof}
It follows from Theorem~\ref{thm:RninftyZninfty} and Theorem~\ref{thm:HGHgeodesic}.
\end{proof}

\markright{References}
\renewcommand\refname{References}

\end{document}